\documentclass[a4paper,12pt,final]{amsart}
\usepackage{times,a4wide,mathrsfs,amssymb,amsmath,amsthm,wasysym}

\newcommand{\C}{\mathbb{C}}
\newcommand{\ZZ}{\mathbb{Z}}

\newcommand{\QQ}{\mathbb{Q}}

\newcommand{\PP}{\mathbb{P}}

\newcommand{\Sy}{\mathfrak S}

\newcommand{\MM}{\mathcal M}

 \makeatletter
\newcommand*{\da@rightarrow}{\mathchar"0\hexnumber@\symAMSa 4B }
\newcommand*{\da@leftarrow}{\mathchar"0\hexnumber@\symAMSa 4C }
\newcommand*{\xdashrightarrow}[2][]{%
  \mathrel{%
    \mathpalette{\da@xarrow{#1}{#2}{}\da@rightarrow{\,}{}}{}%
  }%
}
\newcommand{\xdashleftarrow}[2][]{%
  \mathrel{%
    \mathpalette{\da@xarrow{#1}{#2}\da@leftarrow{}{}{\,}}{}%
  }%
}
\newcommand*{\da@xarrow}[7]{%
  \sbox0{$\ifx#7\scriptstyle\scriptscriptstyle\else\scriptstyle\fi#5#1#6\m@th$}%
  \sbox2{$\ifx#7\scriptstyle\scriptscriptstyle\else\scriptstyle\fi#5#2#6\m@th$}%
  \sbox4{$#7\dabar@\m@th$}%
  \dimen@=\wd0 %
  \ifdim\wd2 >\dimen@
    \dimen@=\wd2 %
  \fi
  \count@=2 %
  \def\da@bars{\dabar@\dabar@}%
  \@whiledim\count@\wd4<\dimen@\do{%
    \advance\count@\@ne
    \expandafter\def\expandafter\da@bars\expandafter{%
      \da@bars
      \dabar@ 
    }%
  }%
  \mathrel{#3}%
  \mathrel{%
    \mathop{\da@bars}\limits
    \ifx\\#1\\%
    \else
      _{\copy0}%
    \fi
    \ifx\\#2\\%
    \else
      ^{\copy2}%
    \fi
  }%
  \mathrel{#4}%
}
\makeatother

\DeclareMathOperator{\aut}{Aut}

\DeclareMathOperator{\ide}{id}

\newtheorem{theorem}{Theorem}[section]

\newtheorem{corollary}[theorem]{Corollary}
\newtheorem{proposition}[theorem]{Proposition}
\newtheorem{conjecture}[theorem]{Conjecture}
\newtheorem{remark}[theorem]{Remark}
\newtheorem{definition}[theorem]{Definition}
\newtheorem{convention}{Conventions}

\newtheorem{nonumbering}{Theorem}

\newtheorem{nonumberingc}{Corollary}

\newtheorem{nonumberingt}{Acknowledgements}

\begin{document}
\author[Robert Laterveer]
{Robert Laterveer}

\address{Institut de Recherche Math\'ematique Avanc\'ee,
CNRS -- Universit\'e 
de Strasbourg,\
7 Rue Ren\'e Des\-car\-tes, 67084 Strasbourg CEDEX,
FRANCE.}
\email{robert.laterveer@math.unistra.fr}

\title{Zero-cycles on Garbagnati surfaces}

\begin{abstract} Garbagnati has constructed certain surfaces of general type that are bidouble planes as well as double covers of K3 surfaces. In this note, we study the Chow groups (and Chow motive) of these surfaces.
 \end{abstract}


\maketitle

\section{Introduction}

 In \cite{Gar}, Garbagnati studies K3 surfaces which admit a double cover. In the course of this work, she constructs certain surfaces of general type that are birational to bidouble covers of $\PP^2$ and that we propose to call {\em Garbagnati surfaces\/}. These surfaces have the remarkable property of being 
 ``K3 burgers'', i.e. they are of geometric genus $m:=p_g(S)\in\{1,2,3\}$ and their transcendental cohomology splits
   \[ H^2_{tr}(S,\QQ)=\bigoplus_{j=1}^m H^2_{tr}(T_j,\QQ)\ ,\]
   where the $T_j$ are K3 surfaces. (For other examples of K3 burgers, cf. \cite{CP}, \cite{triple}.) 
   
  The main result of the present note is that this ``K3 burger'' relation is also valid on the level of Chow groups (and Chow motives):

 \begin{nonumbering}[=Theorem \ref{main}] Let $S$ be a Garbagnati surface. 
 Let $T_1,\ldots,T_m$ be the associated K3 surfaces where $m=p_g(S)$. There is an isomorphism of Chow groups
   \[   A^2_{hom}(S)\  \xrightarrow{\cong}\ \bigoplus_{j=1}^m A^2_{hom}(T_j)  \ .\]
   Moreover, there is an isomorphism of Chow motives
   \[  h^2_{tr}(S)\ \cong\   \bigoplus_{j=1}^m h^2_{tr}(T_j)\ \ \ \hbox{in}\ \MM_{\rm rat}\ .\]
   (Here, $h^2_{tr}()$ denotes the transcendental part of the motive, cf. \S \ref{ss:tr}.)
  \end{nonumbering}
 
 This is in agreement with the Bloch--Beilinson conjectures \cite{J}. 
 As a consequence of the main result, one can verify an old conjecture of Voisin \cite{V9} for certain Garbagnati surfaces:
  
  \begin{nonumberingc}[=Corollary \ref{cor0}]  Let $S$ be a Garbagnati surface of type G1, G2b or G3 (cf. \S \ref{ss:G}).
  Let $n$ be an integer strictly larger than the geometric genus $p_g(S)$. Then for any degree $0$ $0$-cycles $a_1,\ldots,a_n\in A^2_{}(S)_{\ZZ}$, one has
  \[ \sum_{\sigma\in\Sy_n} \hbox{sgn}(\sigma) a_{\sigma(1)}\times\cdots\times a_{\sigma(n)}=0\ \ \ \hbox{in}\ A^{2n}(S^n)_{\ZZ}\ .\]
  (Here $\Sy_n$ is the symmetric group on $n$ elements, and $ \hbox{sgn}(\sigma)$ is the sign of the permutation $\sigma$.
  The notation $a_1\times\cdots\times a_n$ is shorthand for the $0$-cycle $(p_1)^\ast(a_1)\cdot (p_2)^\ast(a_2)\cdots (p_n)^\ast(a_n)$ on 
  $S^n$, where the $p_j\colon S^n\to S$ are the various projections.)
   \end{nonumberingc}
   
   For surfaces of type G1, this was already proven in \cite[Proposition 29]{16.5}, but the present argument seems easier and more conceptual.
   
  Another consequence is a ``motivic Torelli theorem'' for Garbagnati surfaces:
  
  \begin{nonumberingc}[=Corollary \ref{cor1}] Let $S$ and $S^\prime$ be two Garbagnati surfaces, and assume $S$ and $S^\prime$ are isometric (i.e., there is an isomorphism of $\QQ$-vector spaces $H^2_{tr}(S,\QQ)\cong H^2_{tr}(S^\prime,\QQ)$ compatible with Hodge structures and cup product). Then there is an isomorphism of Chow motives
    \[ h^2_{tr}(S)\ \cong\ h^2_{tr}(S^\prime)\ \ \ \hbox{in}\ \MM_{\rm rat}\ .\]
  \end{nonumberingc}
    
  Other consequences are Kimura finite-dimensionality for certain Garbagnati surfaces (Corollary \ref{cor2}), and a ``relative Bloch conjecture'' type of result concerning the action of certain automorphisms (Corollary \ref{cor1.5}). 

 \vskip0.6cm

\begin{convention} In this article, the word {\sl variety\/} will refer to a reduced irreducible scheme of finite type over $\C$. A {\sl subvariety\/} is a (possibly reducible) reduced subscheme which is equidimensional. 

{\bf All Chow groups will be with rational coefficients}: we will denote by $A_j(X)$ the Chow group of $j$-dimensional cycles on $X$ with $\QQ$-coefficients; for $X$ smooth of dimension $n$ the notations $A_j(X)$ and $A^{n-j}(X)$ are used interchangeably. 

The notations $A^j_{hom}(X)$, $A^j_{AJ}(X)$ will be used to indicate the subgroups of homologically trivial, resp. Abel--Jacobi trivial cycles.
For a morphism $f\colon X\to Y$, we will write $\Gamma_f\in A_\ast(X\times Y)$ for the graph of $f$.
The contravariant category of Chow motives (i.e., pure motives with respect to rational equivalence as in \cite{Sc}, \cite{MNP}) will be denoted $\MM_{\rm rat}$.
\end{convention}

\section{Preliminaries}

\subsection{Garbagnati surfaces}
\label{ss:G}

\begin{definition}[\cite{Gar}]\label{def} 
A surface of type G1 is a surface $X$ as in \cite[Proposition 6.1]{Gar}. A surface if type G1 has $p_g=1$, and is birational to a bidouble cover of $\PP^2$ branched along two cubics and a line.

A surface of type G2a is a surface $X_6^{(1)}$ as in \cite[Section 5.4]{Gar}. A surface of type G2a has $p_g=2$, and is birational to a bidouble cover of $\PP^2$ branched along a quintic and two lines.

A surface of type G2b is a surface $X_9^{(2)}$ as in \cite[Section 5.4]{Gar}. A surface of type G2b has $p_g=2$, and is birational to a bidouble cover of $\PP^2$ branched along a quartic and two quadrics.

A surface of type G3 is a surface $Z$ as in \cite[Proposition 6.1]{Gar}. A surface of type G3 has $p_g=3$, and is birational to a bidouble cover of $\PP^2$ branched along three cubics.
\end{definition}

\begin{remark} For details on the construction of bidouble covers, cf. \cite{Cat2}.

The bidouble covers birational to surfaces of type G1 are also known as {\em special Kunev surfaces}, and are studied (particularly in relation to Torelli problems, where they provide counterexamples) in \cite{Cat}, \cite{Kun}, \cite{Tod}.

The bidouble covers birational to surfaces of type G2a are also called {\em special Horikawa surfaces\/}. These surfaces are studied in \cite{PZ}, where it is shown they satisfy generic global Torelli.
\end{remark}

\subsection{Quotient varieties}
\label{ssq}

\begin{definition} A {\em projective quotient variety\/} is a variety
  \[ X=Y/G\ ,\]
  where $Y$ is a smooth projective variety and $G\subset\hbox{Aut}(Y)$ is a finite group.
  \end{definition}
  
 \begin{proposition}[Fulton \cite{F}]\label{quot} Let $X$ be a projective quotient variety of dimension $n$. Let $A^\ast(X)$ denote the operational Chow cohomology ring with $\QQ$-coefficients. The natural map
   \[ A^i(X)\ \to\ A_{n-i}(X) \]
   is an isomorphism for all $i$.
   \end{proposition}
   
   \begin{proof} This is \cite[Example 17.4.10]{F}.
      \end{proof}

\begin{remark} It follows from Proposition \ref{quot} that the formalism of correspondences goes through unchanged for projective quotient varieties (this is also noted in \cite[Example 16.1.13]{F}). We can thus consider motives $(X,p,0)\in\MM_{\rm rat}$, where $X$ is a projective quotient variety and $p\in A^n(X\times X)$ is a projector. For a projective quotient variety $X=Y/G$, one readily proves (using Manin's identity principle) that there is an isomorphism
  \[  h(X)\cong h(Y)^G:=(Y,\Delta^G_Y,0)\ \ \ \hbox{in}\ \MM_{\rm rat}\ ,\]
  where $\Delta^G_Y$ denotes the idempotent ${1\over \vert G\vert}{\sum_{g\in G}}\Gamma_g$.  
  \end{remark}

 \subsection{Transcendental part of the motive}
 \label{ss:tr}
 
 \begin{theorem}[Kahn--Murre--Pedrini \cite{KMP}]\label{t2}    Let $S$ be any smooth projective surface, and let $h(X)\in\MM_{\rm rat}$ denote the Chow motive of $S$.
There exists a self-dual Chow--K\"unneth decomposition $\{\pi^i_S\}$ of $S$, with the property that there is a further splitting in orthogonal idempotents
  \[ \pi^2_S= \pi^{2,alg}_{S}+\pi^{2,tr}_{S}\ \ \hbox{in}\ A^2(S\times S)_{}\ .\]
  The action on cohomology is
  \[  (\pi^{2,alg}_{S})_\ast H^\ast(S,\QQ)= N^1 H^2(S,\QQ)\ ,\ \ (\pi^{2,tr}_{S})_\ast H^\ast(S,\QQ) = H^2_{tr}(S,\QQ)\ ,\]
  where the transcendental cohomology $H^2_{tr}(S,\QQ)\subset H^2(S,\QQ)$ is defined as the orthogonal complement of $N^1 H^2(S,\QQ)$ with respect to the intersection pairing. The action on Chow groups is
  \[ (\pi^{2,alg}_{S})_\ast A^\ast(S)_{}= N^1 H^2(S,\QQ)\ ,\ \ (\pi^{2,tr}_{S})_\ast A^\ast(S) = A^2_{AJ}(S)_{}\ .\]  
 This gives rise to a well-defined Chow motive
  \[ h_{tr}^2(S):= (S,\pi^{2,tr}_{S},0)\ \subset \ h(X)\ \ \in\MM_{\rm rat}\ ,\]
  the so-called {\em transcendental part of the motive of $S$}.
  \end{theorem}

\begin{proof} Let $\{\pi^i_S\}$ be a Chow--K\"unneth decomposition as in \cite[Proposition 7.2.1]{KMP}. The assertion then follows from \cite[Proposition 7.2.3]{KMP}. The projector $\pi^{2,alg}_S$ is of the form
  \[  \pi^{2,alg}_S= \sum_{j=1}^\rho  D_j\times D_j^\prime\ ,\]
  where $D_1,\ldots,D_\rho$ is a basis for the N\'eron--Severi group of $S$, and $D_1^\prime,\ldots, D_\rho^\prime$ is a dual basis.
\end{proof}

 \section{Main result}
 
 \begin{theorem}\label{main} Let $S$ be a Garbagnati surface. 
 Let $T_1,\ldots,T_m$ be the associated K3 surfaces where $m=p_g(S)$. There is an isomorphism
   \[   A^2_{hom}(S)\  \xrightarrow{\cong}\ \bigoplus_{j=1}^m A^2_{hom}(T_j)  \ .\]
   
 Moreover, one can choose Chow--K\"unneth decompositions for $S$ and $T_j$ such that there is an isomorphism of Chow motives
   \[ h^2_{tr}(S)\ \cong\   \bigoplus_{j=1}^m h^2_{tr}(T_j)\ \ \ \hbox{in}\ \MM_{\rm rat}\ .\]
    \end{theorem}
   
 \begin{proof} The statement being birationally invariant, we may assume $S$ is a bidouble plane as described in Definition \ref{def}. From the construction of $S$ as a bidouble plane, it follows that there are three commuting involutions $\sigma_i\in\aut(S)$, $i=1,2,3$, such that 
   \[\sigma_3=\sigma_1\circ\sigma_2=\sigma_2\circ\sigma_1\ .\]
 Let $Q_i:=S/\langle\sigma_i\rangle$ denote the quotients, so there is a commuting diagram (where the arrows are the various quotient maps)
  \[ \begin{array}[c]{ccccc}
    &&S&&\\
    &\swarrow&\downarrow&\searrow&\\
    Q_1&&Q_2&&Q_3\\
     &\searrow&\downarrow&\swarrow&\\ 
     &&\PP^2&&\\
     \end{array}\]
   Define $A^2(S)^{\pm \mp}$ as the subgroup of $A^2(S)$ where $(\sigma_1,\sigma_2)$ acts as $(\pm1,\mp1)$. There is a decomposition
   \[ A^2_{hom}(S)= A^2_{hom}(S)^{++}\oplus    A^2_{hom}(S)^{+-}  \oplus A^2_{hom}(S)^{-+}\oplus  A^2_{hom}(S)^{--} \ .\]
   The summand $A^2_{hom}(S)^{++}$, being isomorphic to 
   $A^2_{hom}(\PP^2)$, is zero, while the three other summands are isomorphic to $A^2_{hom}(Q_i)$, $i=1,2,3$. Thus we get a natural isomorphism
    \[ A^2_{hom}(S)\ \cong \      A^2_{hom}(Q_1)^{}  \oplus A^2_{hom}(Q_2)^{}\oplus  A^2_{hom}(Q_3)^{} \ .\]
    
    In case the surface $S$ is of type G3, there are 3 associated K3 surfaces $T_j$ and one has natural isomorphisms
      \[ A^2_{hom}(Q_j)\ \cong\ A^2_{hom}(T_j)\] 
      (indeed, the $T_j$ are resolutions of singularities of the $Q_j$, and cyclic quotient singularities
    can be resolved by strings of rational curves, cf. \cite{BPV}; the isomorphism then follows from \cite[Proposition 1.7]{42}). This proves the statement for Chow groups.
    
  In case $S$ is of type G2a or G2b, two of the $Q_j$ (say $Q_1$ and $Q_2$) are birational to a K3 surface, and the third surface $Q_3$ is rational, hence $A^2_{hom}(Q_3)=0$. This proves the statement for Chow groups for cases G2a, G2b.
    
 Finally, for the case G1, one of the $Q_j$ (say $Q_1$) is birational to a K3 surface, and the other two are rational; the statement for Chow groups follows similarly.
 
 The statement for motives is proven along the same lines, by exploiting the bidouble cover structure: we can define motives $h(S)^{\pm \mp}\in\MM_{\rm rat}$ by setting
   \[  \begin{split}
                 h(S)^{++}&:=(S,{1\over 4}(\Delta_S+\Gamma_{\sigma_1})\circ(\Delta_S+\Gamma_{\sigma_2}),0)\ ,\\
                  h(S)^{+-}&:=(S,{1\over 4}(\Delta_S+\Gamma_{\sigma_1})\circ(\Delta_S-\Gamma_{\sigma_2}),0)\ ,\\ 
                  h(S)^{-+}&:=(S,{1\over 4}(\Delta_S-\Gamma_{\sigma_1})\circ(\Delta_S+\Gamma_{\sigma_2}),0)\ ,\\ 
                  h(S)^{--}&:=(S,{1\over 4}(\Delta_S-\Gamma_{\sigma_1})\circ(\Delta_S-\Gamma_{\sigma_2}),0)\ .\\ 
                  \end{split}\]
         (It is readily checked the given cycles are idempotents and so define motives.)
                  
  This gives a decomposition
    \[ h(S)=h(S)^{++}\oplus h(S)^{+-}\oplus h(S)^{-+}\oplus h(S)^{--}\ \ \ \hbox{in}\ \MM_{\rm rat}\ .\]                
  Defining $h^0(S)$ and $h^4(S)$ by the choice of a zero-cycle invariant under $\sigma_1$ and $\sigma_2$, we get a similar decomposition for $h^2(S)=h(S)-h^0(S)-h^4(S)$. Next, we can choose a basis $D_1,\ldots,D_\rho$ for the N\'eron--Severi group of $S$ such that
 $D_1,\ldots,D_{\rho_1}$ are of type $++$ (i.e. they are invariant under $\sigma_1$ and $\sigma_2$), $D_{\rho_1+1},\ldots, D_{\rho_2}$ are of type $+-$,  $D_{\rho_2+1},\ldots, D_{\rho_3}$ are of type $-+$ , and the remaining divisors are of type $--$. The dual basis decomposes similarly (a divisor $D_j$ and its dual $D_j^\prime$ are of the same type), and so we get a decomposition
  \[ h^2_{alg}(S)=h^2_{alg}(S)^{++}\oplus h^2_{alg}(S)^{+-}\oplus h^2_{alg}(S)^{-+}\oplus h^2_{alg}(S)^{--}\ \ \ \hbox{in}\ \MM_{\rm rat}\ .\]    
 It follows that there is an induced decomposition
          \[ h^2_{tr}(S)=h^2_{tr}(S)^{++}\oplus h^2_{tr}(S)^{+-}\oplus h^2_{tr}(S)^{-+}\oplus h^2_{tr}(S)^{--}\ \ \ \hbox{in}\ \MM_{\rm rat}\ .\]       
          The first summand $h^2_{tr}(S)^{++}$ is the transcendental part of the motive of $\PP^2$, which is zero. For surfaces $S$ of type G3, the remaining three summands are isomorphic to $h^2_{tr}(S/\langle\sigma_j\rangle)=h^2_{tr}(T_j)$ where $T_j$ is an associated K3 surface. In cases G2a and G2b, the  
          summand $h^2_{tr}(S)^{--}$ corresponds to $h^2_{tr}(Q_3)$, which is zero as $Q_3$ is a
           rational surface. Finally, in case G1 there are two rational surfaces $Q_2$ and $Q_3$ and thus 
           \[   h^2_{tr}(S)=h^2_{tr}(S)^{+-}=h^2_{tr}(Q_1)=h^2_{tr}(T_1)\ \ \ \hbox{in}\ \MM_{\rm rat}\ .\]
           This proves the motivic statement.
                   \end{proof}

 \section{Some consequences}

 \subsection{Voisin conjecture}
 
 \begin{conjecture}[Voisin \cite{V9}]\label{conj} Let $S$ be a smooth projective surface. Let $n$ be an integer strictly larger than the geometric genus $p_g(S)$. Then for any $0$-cycles $a_1,\ldots,a_n\in A^2_{AJ}(S)_{\ZZ}$, one has
  \[ \sum_{\sigma\in\Sy_n} \hbox{sgn}(\sigma) a_{\sigma(1)}\times\cdots\times a_{\sigma(n)}=0\ \ \ \hbox{in}\ A^{2n}(S^n)_{\ZZ}\ .\]
  (Here $\Sy_n$ is the symmetric group on $n$ elements, and $ \hbox{sgn}(\sigma)$ is the sign of the permutation $\sigma$.
  The notation $a_1\times\cdots\times a_n$ is shorthand for the $0$-cycle $(p_1)^\ast(a_1)\cdot (p_2)^\ast(a_2)\cdots (p_n)^\ast(a_n)$ on 
  $S^n$, where the $p_j\colon S^n\to S$ are the various projections.)
 \end{conjecture}
 
This conjecture is a particular instance of the Bloch--Beilinson conjectures.
For surfaces of geometric genus $0$, conjecture \ref{conj} reduces to Bloch's conjecture \cite{B}. As for geometric genus $1$, Voisin's conjecture is still open for a general K3 surface;
examples of surfaces of geometric genus $1$ verifying the conjecture are given in \cite{V9}, \cite{16.5}, \cite{19}, \cite{21}. Examples
of surfaces with geometric genus strictly larger than $1$ verifying the conjecture are given in \cite{32}, \cite{33}, \cite{LV}.

 \begin{corollary}\label{cor0} Let $S$ be a surface of type G1, G2b or G3. Then Voisin's conjecture is true for $S$.
  \end{corollary}
  
  \begin{proof} Thanks to Roitman's theorem, Voisin's conjecture is equivalent to the version with $\QQ$-coefficients.
  Voisin's conjecture (for an arbitrary surface $S$) can then be succinctly restated as
    \[  A_0\bigl(\wedge^n h^2_{tr}(S)\bigr)=0\ \ \ \hbox{for\ all\ }n>p_g(S)\ ,\]  
    where the wedge product of a motive is as defined in \cite[Definition 3.5]{Kim}. 
    
Let $S$ now be a Garbagnati surface, and let $T_1,\ldots,T_m$ denote the K3 surfaces associated to $S$. Invoking Theorem \ref{main}, we find that Voisin's conjecture for $S$ is implied by Voisin's conjecture for $T_1,\ldots,T_m$. Indeed, we have
  \[ \begin{split}     A_0\bigl(\wedge^n h^2_{tr}(S)\bigr)&=    A_0\Bigl(   \wedge^n \bigl( \bigoplus_{j=1}^m h^2_{tr}(T_j)\bigr)\Bigr)\\
                                                                                   &=  A_0\Bigl( \bigoplus_{n_1+\cdots+n_m=n}  \wedge^{n_1} h^2_{tr}(T_1)\otimes  \cdots\otimes \wedge^{n_m} h^2_{tr}(T_m)   \Bigr)\\
                                                                                   &= \bigoplus_{n_1+\cdots+n_m=n} A_0(  \wedge^{n_1} h^2_{tr}(T_1))\otimes\cdots\otimes A_0(  \wedge^{n_m} h^2_{tr}(T_m)) \ .  \\
                                                                                   \end{split}\]
                                                                    Assume now that $n>m:=p_g(S)$. Then in each summand there is an $n_j>1$. Hence assuming Voisin's conjecture for all the $T_j$, each summand vanishes, and so Voisin's conjecture holds for $S$ as claimed.
                                                                    
To finish the proof, it remains to observe that Voisin's conjecture is known for K3 surfaces obtained by desingularizing double planes branched along the union of two cubics \cite[Theorem 3.4]{V9}, and also for K3 surfaces obtained from double planes branched along a quartic and a quadric \cite[Proposition 14]{16.5}.                                                                                  
 \end{proof}
 
 \begin{remark} We do not know whether Voisin's conjecture holds for surfaces of type G2a. The reason is that Voisin's conjecture is not yet known for
 K3 double planes branched along a quintic and a line (and the method of \cite[Theorem 3.4]{V9} and \cite[Proposition 14]{16.5} seems ill-suited for this case).
 \end{remark}

 \subsection{Motivic Torelli}
 
 \begin{corollary}\label{cor1} Let $S$ and $S^\prime$ be two Garbagnati surfaces, and assume $S$ and $S^\prime$ are isometric (i.e., there is an isomorphism of $\QQ$-vector spaces $H^2_{tr}(S,\QQ)\cong H^2_{tr}(S^\prime,\QQ)$ compatible with Hodge structures and cup product). Then there is an isomorphism of Chow motives
    \[ h^2_{tr}(S)\ \cong\ h^2_{tr}(S^\prime)\ \ \ \hbox{in}\ \MM_{\rm rat}\ .\]
  \end{corollary}
  
  \begin{proof} This is a direct transplantation of the analogous result for K3 surfaces. 
  
  Let $T_j$ and $T_j^\prime$ denote the K3 surfaces associated to $S$ resp. to $S^\prime$, for $j=1,\ldots,m:=p_g(S)=p_g(S^\prime)$. The image of $H^2_{tr}(T_j,\QQ)\subset H^2_{tr}(S,\QQ)$ under the isometry $\phi$ is a sub-Hodge structure of $H^2_{tr}(S^\prime,\QQ)$. Since the transcendental cohomology of K3 surfaces are 
  indecomposable Hodge structures, the image must be equal to one of the $H^2_{tr}(T^\prime_j,\QQ)$, and so we may assume that $\phi\bigl(   H^2_{tr}(T_j,\QQ)\bigr) = H^2_{tr}(T^\prime_j,\QQ)$. The inclusion $H^2_{tr}(T_j,\QQ)\subset H^2_{tr}(S,\QQ)$ is given by the composition
    \[  H^2_{tr}(T_j,\QQ)\ \xleftarrow{\cong}\ H^2_{tr}(S/\langle\sigma_j\rangle,\QQ)\ \xrightarrow{(p_j)^\ast}\ H^2_{tr}(S,\QQ)\ ,\]
    where the left arrow is induced by a resolution of singularities. This shows that the inclusion $H^2_{tr}(T_j,\QQ)\subset H^2_{tr}(S,\QQ)$
    is compatible with cup product. Hence, the isometry $\phi$ decomposes as a sum
    \[ \phi=\sum_{j=1}^m \phi_j\colon\ \ \bigoplus_{j=1}^m H^2_{tr}(T_j,\QQ)\ \to\ \bigoplus_{j=1}^m H^2_{tr}(T^\prime_j,\QQ) \ ,\]
    where each $\phi_j$ is an isometry. Huybrechts' result \cite{Huy2} then guarantees that there are isomorphisms of Chow motives
    \[   h^2_{tr}(T_j)\ \cong\ h^2_{tr}(T^\prime_j)    \ \ \ \hbox{in}\ \MM_{\rm rat}\ \ \ (j=1,\ldots,m)\ .\]
    Applying Theorem \ref{main}, it follows that there is an isomorphism
    \[ h^2_{tr}(S)\ \cong\ h^2_{tr}(S^\prime)    \ \ \ \hbox{in}\ \MM_{\rm rat}    \ .\]
     \end{proof}

 \subsection{Finite-dimensionality}
 
 \begin{corollary}\label{cor2} The following surfaces have finite-dimensional motive (in the sense of Kimura \cite{Kim}):
 
 \begin{enumerate}
 
 \item Surfaces of type G1 with $\dim H^2_{tr}(S,\QQ)\le 3$;
 
 \item Surfaces of type G2 with $\dim H^2_{tr}(S,\QQ)\le 5$; 
 
 \item Surfaces of type G3 with $\dim H^2_{tr}(S,\QQ)\le 7$.
 \end{enumerate}
 \end{corollary}
 
 \begin{proof}
 Let $m\in\{ 1,2,3\}$ denote the geometric genus $m:=p_g(S)$, and let 
 $T_1,\ldots,T_m$ be the associated K3 surfaces. Recall that there is an isomorphism
     \[ H^2_{tr}(S,\QQ)\cong \oplus_{j=1}^m H^2_{tr}(T_j,\QQ)\ .\]  
   The $T_j$ being K3 surfaces, the dimension of $H^2_{tr}(T_j,\QQ)$ is at least $2$, and so the assumption on $ H^2_{tr}(S,\QQ)$ implies that
   \[ \dim  H^2_{tr}(T_j,\QQ)\le 3\ \ \ (j=1\ldots,m)\ .\]
   It follows from \cite{Ped} that the $T_j$ have finite-dimensional motive. In view of the isomorphism of Theorem \ref{main}, this implies the corollary.
 \end{proof}

%
%
%
%
 
 \subsection{Bloch conjecture for automorphisms}
 
 \begin{corollary}\label{cor1.5} Let $S$ be a Garbagnati surface with $p_g(S)=:m$, and let $\sigma_1,\ldots,\sigma_m$ be the involutions for which the quotient is birational to a K3 surface. Let $f\in\aut(S)$ be a
 finite order automorphism that commutes with $\sigma_1,\ldots,\sigma_m$, and such that
   \[ f^\ast=\ide\colon\ \ \ H^{2,0}(S)\ \to\ H^{2,0}(S)\ .\]
   Then also
    \[ f^\ast=\ide\colon\ \ \ A^{2}_{}(S)\ \to\ A^{2}(S)\ .\] 
   \end{corollary}

   \begin{proof} Since $f$ commutes with the $\sigma_j$, $f$ induces automorphisms $f_j\in\aut(T_j), j=1,\ldots,m$ that are symplectic of finite order. Huybrechts has proven \cite{Huy} that such automorphisms act as the identity on zero-cycles, i.e. one has
    \[  (f_j)^\ast=\ide\colon\ \ \ A^2(T_j)\ \to\ A^2(T_j)\ \ \ (j=1,\ldots,m)\ .\]
    Theorem \ref{main}, combined with the commutative diagram
     \[ \begin{array}[c]{ccc}
        A^2_{hom}(S) & \xrightarrow{ f^\ast} & A^2_{hom}(S) \\
        \uparrow{\scriptstyle (p_j)^\ast}&&  \uparrow{\scriptstyle (p_j)^\ast}\\
        A^2_{hom}(T_j) & \xrightarrow{ (f_j)^\ast} & A^2_{hom}(T_j) \\ 
        \end{array} \ \ \ \ \ \ (j=1,\ldots,m) \]
        implies that
       \[ f^\ast=\ide\colon\ \ \ A^{2}_{hom}(S)\ \to\ A^{2}_{hom}(S)\ .\]   
       Since the subspace $A^2(S)^{++}\cong\QQ$ is also fixed by $f$, this proves the corollary.
      \end{proof}

\vskip0.6cm
\begin{nonumberingt} Thanks to Len and Kai and Yoyo for being wonderful collaborators in the 2019 Special Program ``Papa Werkt Thuis''. 
\end{nonumberingt}

\end{document}